\documentclass[a4paper,11pt]{article}

\usepackage{kpfonts}
\usepackage[utf8]{inputenc}
\usepackage[margin=3cm]{geometry}

\usepackage{graphicx}
\usepackage{subcaption}
\usepackage{amsmath, amssymb, amsthm}
\usepackage{mathtools}

\usepackage{microtype}
\usepackage{booktabs}
\usepackage{enumerate,paralist}
\usepackage{xspace}
\usepackage{comment}
\usepackage{xcolor}
\usepackage{array}
\usepackage[numbers,sort]{natbib}
\definecolor{linkblue}{HTML}{2B4C7E}
\usepackage[pdfpagelabels,colorlinks,allcolors=linkblue]{hyperref}
\usepackage[nameinlink,capitalize,noabbrev,sort]{cleveref}

\theoremstyle{plain}

\newtheorem{theorem}{Theorem}
\newtheorem{lemma}{Lemma}

\newtheorem{fact}{Fact}
\crefname{fact}{Fact}{Facts}
\Crefname{fact}{Fact}{Facts}

\theoremstyle{definition}  
\newtheorem{problem}{Problem}

\graphicspath{{pics/}}

\definecolor{defblue}{rgb}{0.121,0.47,0.705}
\definecolor{defred}{rgb}{0.7,0.37,0.37}
\definecolor{collocal}{rgb}{0.894,0.102,0.110}
\definecolor{defgreen}{rgb}{0.01,0.65,0.20}
\definecolor{brightmaroon}{rgb}{0.76, 0.13, 0.28}

\newcommand{\df}[1]{\textcolor{brightmaroon}{\emph{#1}}}
\newcolumntype{C}[1]{>{\centering\arraybackslash}p{#1}}
\newcommand{\yes}{{\it \textcolor{defgreen}{yes}}}

\newcommand{\oops}{\texttt{OOPS}\xspace}
\newcommand{\NP}{\ensuremath{\mathsf{NP}}\xspace}

\begin{document}

\title{Smallest Cubic Non-1-Planar Graphs}
\author{Sergey Pupyrev\\
\small\texttt{spupyrev@gmail.com}}
\date{}
\maketitle

\begin{abstract}
    A graph is 1-planar if it has a drawing in which every edge is crossed
    at most once. We show that the smallest cubic non-1-planar graphs
    have $30$ vertices. Two such graphs are the \texttt{Tutte-Coxeter}
    graph of girth eight and a graph of girth seven that we call the
    \texttt{Byte} graph. Every subcubic graph with fewer than $30$ vertices
    is 1-planar.

    Our proof is computer-assisted, but directly testing all relevant
    graphs is impractical. To establish non-1-planarity of the two graphs,
    we extend a SAT-based solver with a custom clause propagator based
    on separating cycles and a case split based on graph automorphisms,
    allowing independent cases to be solved in parallel. To show that
    all smaller subcubic graphs are 1-planar, we introduce the concept of $k$-flexibility:
    every set of at most $k$ prescribed edges can remain uncrossed in
    some 1-planar drawing. We use this property to reconstruct 1-planar drawings
    of larger graphs from drawings of smaller $k$-flexible graphs.
    This replaces exhaustive testing
    of more than forty billion cubic graphs with computations on far fewer
    graphs of smaller order.
\end{abstract}

\paragraph{Declaration of generative AI use}
ChatGPT Codex was used to assist with the implementation and with proofreading of
the manuscript.
The author supplied the ideas and a high-level
implementation plan and guided the development of the core optimizations.
The author
verified all AI-generated code, and takes full responsibility for all claims
in the paper.

\section{Introduction}

Planarity is a central concept in graph drawing, but many graphs arising in
applications are nonplanar. This motivates the study of
\emph{beyond-planar} graph classes, which allow crossings subject to certain
restrictions~\cite{KLM17,DLM19}. One of the most natural such classes is the class of
1-planar graphs, introduced by Ringel in the 1960s~\cite{R65}. A graph
is \df{1-planar} if it has a drawing in the plane in which every edge is crossed at most once.
Although 1-planar graphs retain several properties of planar graphs, including
linear edge density, recognizing them is
\NP-complete even under strong restrictions~\cite{GB07,KM13,CM13}.
Consequently, determining whether a particular sparse graph is 1-planar can be surprisingly difficult~\cite{BDM23,oops,MFPR26}.

Finding the \emph{smallest} non-1-planar graph in a family, that is,
one with the fewest vertices, is a natural test for an exact
recognition algorithm. In this paper, we consider perhaps
the most challenging elementary instance of this question: cubic graphs.
These graphs are sparse: an $n$-vertex cubic graph contains only $3n/2$ edges, which is far below the general
1-planar density bound $4n-8$~\cite{BSW84}.  Hence, the general density bound
does not establish non-1-planarity for cubic graphs.
It is not surprising that the following question has appeared in several studies of 1-planarity.

\begin{problem}[\cite{11011110_cubic,mathse_cubic,oops}]
    \label{prob:1}
    What is the smallest cubic non-1-planar graph?
\end{problem}

Eppstein constructed 1-planar drawings of several cubic graphs and suggested that
either the Coxeter graph ($28$ vertices) or the \texttt{Tutte-Coxeter} graph ($30$ vertices) might be
non-1-planar~\cite{11011110_cubic}.
The first extensive computational study
of the problem was performed using the SAT-based Optimized One-Planarity
Solver (\oops)~\cite{oops,oopsGD}.  It showed that the Coxeter graph is
1-planar and conjectured that the \texttt{Tutte-Coxeter} graph is not.  The solver also
verified that every cubic graph with at most $24$ vertices and every cubic
bipartite graph with at most $28$ vertices is 1-planar, along with other
families of cubic graphs; refer to \cref{table:cubic} for details.
These computations
covered about $150$ million graphs and required approximately $1.5$
machine-months, yet they remained far from finding the smallest non-1-planar instance.
Alternative exact solvers for 1-planarity by Binucci, Didimo, and Montecchiani~\cite{BDM23} and
by M\"unch, Fink, Pfretzschner, and Rutter~\cite{MFPR26,FMPR25} use backtracking algorithms that prune
many possible \emph{crossing patterns}, that is, sets of pairs of edges chosen to cross. These algorithms can quickly find
1-planar drawings, but proving non-1-planarity requires ruling out every
possible crossing pattern. Highly symmetric cubic instances, such as
\texttt{Tutte-Coxeter}, produce especially large families of equivalent crossing
patterns, making this search impractical. Direct SAT encodings also struggle
to prove non-1-planarity: proving unsatisfiability is considerably harder
than finding a single 1-planar drawing~\cite{oops}.

We stress that proving the existence of some cubic non-1-planar graphs is
relatively straightforward.
In fact, large random cubic graphs have superlinear crossing number~\cite{DKMW08}, and hence,
most of them are non-1-planar. An explicit
construction can be obtained from \df{cube-connected cycles}.
For $d\geq3$,
replace every vertex of the $d$-dimensional hypercube $Q_d$ by a cycle of
length $d$, whose vertices correspond to the $d$ coordinates, and join
corresponding vertices whenever the original hypercube vertices are adjacent
in that coordinate.  The resulting graph $\mathrm{CCC}_d$ is cubic and contains
$d2^d$ vertices and $3d2^{d-1}$ edges.  S{\'y}kora and
Vr{\v t}o~\cite{SV93} prove that
\[
\operatorname{cr}(\mathrm{CCC}_d)
\geq \frac{4^d}{20}-3(d+1)2^{d-2},
\]
where $\operatorname{cr}$ denotes the \df{crossing number} of a graph.
For $d=8$, this lower bound implies that
$\operatorname{cr}(\mathrm{CCC}_8)\geq1549$, whereas a 1-planar drawing of
$\mathrm{CCC}_8$, which has $3072$ edges, could contain at most
$3072/2=1536$ crossings.
Thus, the 2048-vertex graph $\mathrm{CCC}_8$ is non-1-planar, but this
construction gives little information about the smallest possible order.

In this paper, we resolve \cref{prob:1} and prove that the minimum order is exactly $30$.
For the upper bound, we give computational proofs that the \texttt{Tutte-Coxeter}
graph and another cubic graph (called the \texttt{Byte} graph) are non-1-planar;
see \cref{fig:cubic30-graphs}. The two graphs are
nonisomorphic, having girth eight and seven, respectively. For the lower
bound, we prove that every cubic graph with at most $28$ vertices is
1-planar; in fact, every subcubic graph of order less than $30$ is
1-planar.

Our proof combines two sets of new ideas. To prove the upper bound, we
strengthen \oops with two methods designed for difficult unsatisfiable
instances. First, we modify the underlying SAT solver by introducing a
custom clause propagator, following the general approach of extending SAT
search with problem-specific reasoning~\cite{FNPKSB23}. Our propagator uses
a condition based on separating cycles, inspired by M\"unch et
al.~\cite{MFPR26}, to detect partial assignments that cannot extend to
a 1-planar drawing. It adds clauses that rule out these assignments
in subsequent steps.
Second, we split the problem into cases according to which edges are
crossed and use graph automorphisms to avoid testing equivalent crossing
patterns. The cases can be solved independently on multiple cores.
This method is particularly effective for the edge-transitive
\texttt{Tutte-Coxeter} graph. Both methods are essential for completing these
computations within practical running times.

For the lower bound, directly testing all forty billion cubic graphs of order $28$
is infeasible, even on a powerful multicore machine. We instead introduce the notion of $k$-flexibility: a graph
is \df{$k$-flexible} if every set of at most $k$ prescribed edges can be kept
uncrossed in some 1-planar drawing.  Reduction rules for 3-cycles (triangles),
4-cycles, and 5-cycles allow us to construct drawings of larger graphs from
drawings of smaller graphs in which certain edges remain uncrossed.
These rules reduce the proof to exhaustive computations on smaller
graph families.

The paper is organized as follows.
\Cref{sect:oops} introduces the SAT encoding used by \oops.
\Cref{sect:ub} proves the upper bound, and \cref{sect:lb} establishes
the lower bound. We conclude with open questions.

\section{Preliminaries}
\label{sect:oops}

Throughout the paper, all graphs are finite, simple, and undirected.  The
\df{order} of a graph $G=(V, E)$ is $n=|V|$.  A graph is
\df{cubic} if every vertex has degree three, and it is \df{subcubic} if
every vertex has degree at most three.  A connected graph is
\df{biconnected} if it has no cut vertex.  The \df{girth} of a graph
is the length of its shortest cycle.

A graph is \df{1-planar} if it has a drawing in the plane in which every edge is
crossed at most once.
We write $e\mathbin\times f$ for a crossing between edges $e$ and $f$.
A \df{crossing pattern} in a graph $G$ is a set $M$ of unordered pairs of
independent edges such that no edge occurs in more than one pair.
We say that $M$ \df{avoids} an edge, or a set of edges, if none of those
edges occurs in a pair of $M$. The \df{planarization} of $G$ with respect
to $M$, denoted by $G_M$, is obtained by deleting edges $uv$ and $xy$ for each
pair $\{uv,xy\}\in M$ and adding a new vertex adjacent to $u,v,x,y$.
If $G_M$ is planar, then $G$ has a 1-planar drawing in which every edge
avoided by $M$ is uncrossed~\cite{oops}. Such a drawing can always be augmented
around each crossing between $uv$ and $xy$ by four uncrossed
\df{kite} edges $ux,xv,vy,yu$, following the circular order $u,x,v,y$
of the endpoints.

The computational results in this paper use \oops~\cite{oops}, an exact
SAT-based solver for 1-planarity.  Given a graph $G$, the solver constructs a
Boolean formula that is satisfiable if and only if $G$ is 1-planar.  A
satisfying assignment yields a drawing, while an unsatisfiable formula
proves that no such drawing exists. We next provide a high-level overview
of the SAT encoding.

\begin{figure}[!t]
    \centering
    \begin{subfigure}[t]{0.45\linewidth}
        \centering
        \includegraphics[width=\linewidth]{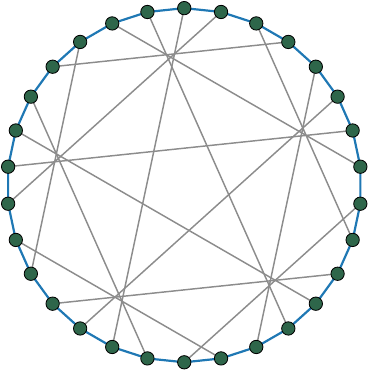}
        \caption*{\texttt{Tutte-Coxeter} graph}
    \end{subfigure}\hfill%
    \begin{subfigure}[t]{0.45\linewidth}
        \centering
        \includegraphics[width=\linewidth]{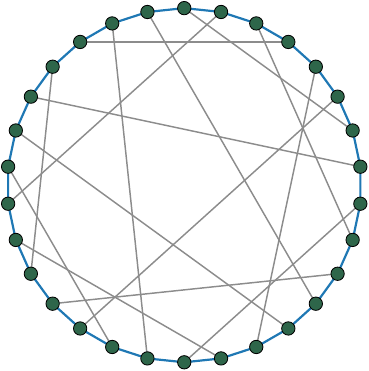}
        \caption*{\texttt{Byte} graph}
    \end{subfigure}
    \caption{The two cubic non-1-planar graphs of order $30$.}
    \label{fig:cubic30-graphs}
\end{figure}

Let $G'$ be obtained from $G$ by subdividing every edge $e$ once, and denote
the \df{division} vertex by $d_e$. A pair of edges $\{e,f\}$ in a crossing pattern
$M$ is represented by \df{merging} $d_e$ and $d_f$. Boolean
variables select these merges, and clauses ensure that every division vertex
is merged with at most one other division vertex. The resulting graph
is the planarization $G_M$ with an additional degree-two division vertex
on each edge avoided by $M$. It is planar if and only if $G_M$ is planar.
To enforce planarity, the encoding uses the fact that this graph is bipartite,
with original and division vertices forming its two parts.  Every planar
bipartite graph admits a two-page \df{stack layout}:
its vertices lie on a common line, called the spine, and its edges
are drawn without crossings in the two half-planes bounded by
the spine~\cite{FFNO11}. Moreover, an edge can be drawn above the
spine if its endpoint in the first part precedes its endpoint in
the second part, and below otherwise.
The encoding therefore introduces Boolean variables $\sigma(u,v)$
indicating whether $u$ precedes $v$ on the spine; two merged
division vertices occupy the same position.  Standard clauses enforce comparability
and transitivity of this order and link coincident positions to the selected
merges.  Finally, for every pair of independent edges of $G'$, clauses forbid
their endpoints from alternating along the spine when the edges lie on the
same page. Thus the selected merges and the stack layout ensure that
$G_M$ is planar, which yields a 1-planar drawing of $G$ as described above.

\oops uses the \texttt{MapleGlucose} SAT solver~\cite{maple}.
The source code of \oops (Optimized One-Planarity Solver) is publicly
available~\cite{oopsCode}.
We run single-core and multi-core computations on a dual-node
server with 3.2 GHz Intel Xeon Platinum 8488C processors and 360 GB RAM.
Multi-core computations use up to 48 cores in parallel.

\section{Two Cubic Graphs of Order 30 Are Non-1-Planar}
\label{sect:ub}

Our goal in this section is to prove computationally that two 30-vertex
cubic graphs are non-1-planar.
The graphs are shown in \cref{fig:cubic30-graphs}. One is
the \texttt{Tutte-Coxeter} graph.
The second one, the \texttt{Byte} graph, is named after its
index ($256$) in the enumeration of
connected cubic graphs of order $30$ and girth at least $7$
produced by \texttt{Minibaum}~\cite{Brinkmann96}.
To show non-1-planarity of the instances, we extend the basic version of
\oops described in \cref{sect:oops} with two optimizations:
dynamic propagation based on separating cycles and a new symmetry-breaking
scheme. Both optimizations are specifically designed for proving unsatisfiability.

\subsection{Separating-Cycle Propagation}
\label{sect:sepcycle-propagation}

The first optimization adds a custom clause propagator to the SAT solver.
A SAT solver searches by assigning truth values to variables, deriving
further assignments from the clauses, and revising earlier choices when
it reaches a contradiction. At an intermediate stage, only some variables
have assigned values.
In our encoding, this partial assignment imposes crossing constraints while
leaving some variables undecided. The new propagator checks whether these constraints
violate a necessary condition for 1-planarity. If they do, it adds a
\df{conflict clause} that excludes the conflicting combination of constraints. This allows the
solver to reject the current assignment and avoid the same conflict later.
Our condition is based on the separating-cycle observation of
M\"unch et al.~\cite{MFPR26}.

Consider a 1-planar drawing of a graph and a partial assignment
of crossings.
Suppose that edges $uv$ and $xy$ cross, and consider a cycle, $C$,
containing $uv$ but neither $x$ nor $y$; see \cref{fig:sepcycle-propagation}.
We call $C$ a \df{separating cycle}, since it separates $x$
and $y$.
For any $x$--$y$ path, $P$, disjoint
from vertices of the cycle, $V(C)$, and not using $xy$, the closed
curves $P\cup xy$ and $C$
cross an even number of times, even if they cross themselves. Since $xy$
crosses $C$ exactly once, $P$ must also cross $C$.
Distinct edge-disjoint paths require distinct crossed
edges of $C$, as each edge can be crossed at most once.

The propagator searches for separating cycles and many edge-disjoint paths
that must cross them. A cycle can accommodate at most as many crossings
as it has edges; the partial assignment of crossings may further
limit that number. If such paths require more
crossings than the cycle can accommodate, the configuration cannot extend
to a 1-planar drawing and we can safely add a corresponding clause to the formula.

\begin{figure}[!t]
    \centering
    \includegraphics[width=0.8\linewidth]{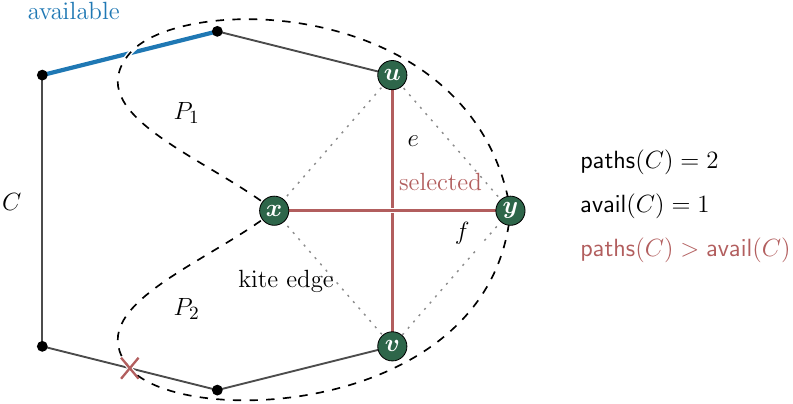}
    \caption{A separating cycle with insufficient crossing capacity. The assignment selects the
        crossing $e\mathbin\times f$, and the dotted curves are its virtual kite
        edges.  Apart from $f$, two edge-disjoint $x$--$y$ paths are forced to
        cross $C$, whereas only the blue edge of $C$ is available for an
        additional crossing. Thus $\mathsf{paths}(C)=2>1=\mathsf{avail}(C)$, and the assignment cannot be extended to a 1-planar drawing.}
    \label{fig:sepcycle-propagation}
\end{figure}

We now formalize the intuition.  A crossing is \df{selected} if the assignment
requires the corresponding pair of edges to cross.
Around each
selected crossing, we add the four corresponding uncrossed kite edges.
Whenever a new crossing is selected, the propagator searches for separating
cycles in this augmented graph, considering each
of the two crossing edges as the edge contained in the cycle. For a
selected crossing between $e=uv$ and $f=xy$, a candidate cycle $C$ contains
$e$ but avoids $x$ and $y$. An edge is \df{available} if it is neither part
of a selected crossing, nor a kite edge, nor fixed to be uncrossed.
Let $\mathsf{avail}(C)$ be the number of available edges on $C$.

For each candidate cycle $C$, we count paths that require additional
crossings with $C$. Start with the original graph $G$, without kite edges,
and delete $V(C)$ and every edge whose selected crossing partner lies on
$C$, including $f$. Let $\mathsf{paths}(C)$ be the maximum number of pairwise
edge-disjoint $x$--$y$ paths in the remaining graph, computed by a
unit-capacity flow calculation. Each path must cross $C$, and the deleted
edges exclude all previously selected crossings with $C$. Thus each
counted path requires an additional crossing.

Since every available edge of $C$ can accommodate at most one additional
crossing, extending the assignment requires $\mathsf{paths}(C)\leq \mathsf{avail}(C)$. If this
inequality is violated, as in \cref{fig:sepcycle-propagation}, the propagator returns a
conflict clause that rules out the selected crossings and uncrossed-edge
constraints responsible for the violation. Informally, at least one of
the selected crossings involved in $C$ must be abandoned, or at least one
edge of $C$ fixed to be uncrossed must instead be crossed. The solver
retains this clause after revising its choices, preventing the same
conflicting combination from occurring later. The test therefore adds
clauses as needed during the search, rather than including all such
constraints in the initial formula.

While the high-level idea of our separating-cycle propagator resembles
the filter of M\"unch et al.~\cite{MFPR26}, we introduce two changes to
make it effective within SAT search. First, we turn each detected conflict
into a clause, allowing the solver to exclude the same conflicting choices
elsewhere without repeating the geometric check. Second, their filter
chooses paths that minimize the number of available cycle edges.
However, the number of paths forced to cross a cycle also depends on
which cycle is chosen. We therefore examine multiple candidate cycles, restricting
their length to limit the cost of invoking the propagator during
the search.

\subsection{Symmetry Breaking}
\label{sect:symmetry-breaking}

The second optimization uses graph symmetries to avoid examining
equivalent crossing patterns. Recall that an \df{automorphism} of a graph
is a permutation of its vertices that preserves adjacency. Applying an
automorphism to a 1-planar drawing gives another drawing of the same
graph, with correspondingly relabeled edges. Thus crossing
patterns related by an automorphism need not be tested separately.
While \oops supports the general-purpose symmetry-breaking tools
\texttt{Satsuma} and \texttt{BreakID}, these yield almost
no performance improvement on our two instances~\cite{oops}.
This is because such tools detect
symmetries of the SAT formula, which need not preserve all symmetries
of the input graph. For example, the SAT encoding of 1-planarity
fixes one vertex as
the first on the spine~\cite{oops}, so a graph automorphism moving
this vertex will not preserve the formula.
To exploit graph
symmetries directly, we instead use automorphisms of the input graph
to split the problem into cases before invoking the SAT solver.

We illustrate the idea on a cubic graph with 30 vertices containing 45 edges.
In any 1-planar drawing, the number of crossed edges is even.
Thus at least one edge is uncrossed.
The \texttt{Tutte-Coxeter} graph is edge-transitive: for any two edges,
an automorphism maps one to the other. We can therefore choose
any one edge and require it to be uncrossed, since every 1-planar
drawing can be relabeled to satisfy this requirement. We apply
the same idea repeatedly, using only automorphisms that preserve
the constraints already imposed.

Formally, each case is described by a pair $(N,P)$ of disjoint edge
sets: the edges in $N$ are fixed to be uncrossed and those in $P$ are
fixed to be crossed. Initially, $N=P=\emptyset$.
Let $\mathrm{Aut}(G)$ denote the automorphism group
of $G$. The \df{stabilizer} $H=\mathrm{Stab}(N)$ is the subgroup of
$\mathrm{Aut}(G)$ consisting of automorphisms that map each edge of $N$
to itself. The \df{orbit} of an edge under $H$ is the set of edges to
which automorphisms in $H$ map it. We maintain the invariant that $P$
is a union of $H$-orbits, so every automorphism in $H$ preserves
the constraints already imposed.
Let $\mathcal O_0,\ldots,\mathcal O_{k-1}$ be the $H$-orbits of
the remaining edges, those outside $N\cup P$, in a fixed order.
Choose a representative $r_i$ from each $\mathcal O_i$.
We divide the current case into $k+1$ cases:
\begin{itemize}
    \item every remaining edge is crossed; or
    \item for each $i\in\{0,\ldots,k-1\}$, all edges in
    $\mathcal O_0\cup\cdots\cup\mathcal O_{i-1}$ are crossed and $r_i$ is
    uncrossed.
\end{itemize}

These cases cover all drawings up to an automorphism.  Consider a
drawing satisfying the constraints at $(N,P)$.  If every remaining edge is
crossed, it belongs to the first case.  Otherwise, let $\mathcal O_i$ be the
first orbit containing an uncrossed edge $e$.  An automorphism in $H$ maps
$e$ to $r_i$. It fixes every edge of $N$ and preserves $P$ and all preceding
orbits, since these are unions of $H$-orbits.
The relabeled drawing therefore
belongs to case $i$. We add $r_i$ to $N$ and all preceding orbits to $P$,
then repeat the split using the new stabilizer $\mathrm{Stab}(N)\subseteq H$.
The updated $P$ remains a union of orbits under this subgroup.

\begin{table}[!t]
    \small
    \centering
    \caption{Wall-clock running times for basic \oops, separating-cycle
        propagation (\cref{sect:sepcycle-propagation}), symmetry-based case
        splitting (\cref{sect:symmetry-breaking}), and their combination.
        All configurations use $48$ cores; $>7$ days indicates that the
        computation did not finish within one week on the multi-core machine.}
    \label{tab:cubic30-runtime}
    \setlength{\tabcolsep}{4pt}
    \begin{tabular*}{\textwidth}{@{\extracolsep{\fill}}lrrrrrrr@{}}
        \toprule
        graph & $|V|$ & $|E|$ & $|\mathrm{Aut}(G)|$
        & basic \oops & \cref{sect:sepcycle-propagation}
        & \cref{sect:symmetry-breaking} & both \\
        \midrule
        \texttt{Robertson} & $19$ & $38$ & $24$
        & $176$ s & $24$ s & $86$ s & $28$ s \\
        \texttt{Folkman} & $20$ & $40$ & $3840$
        & $161$ s & $66$ s & $48$ s & $31$ s \\
        \texttt{Brinkmann} & $21$ & $42$ & $14$
        & $299$ s & $31$ s & $188$ s & $37$ s \\
        \texttt{Holt} & $27$ & $54$ & $54$
        & $1661$ s & $98$ s & $1040$ s & $60$ s \\
        \midrule
        \texttt{Tutte-Coxeter} & $30$ & $45$ & $1440$
        & $>7$ days
        & $>7$ days
        & $>7$ days & $2.6$ h \\
        \texttt{Byte} & $30$ & $45$ & $16$
        & $>7$ days
        & $>7$ days
        & $>7$ days & $6.5$ h \\
        \bottomrule
    \end{tabular*}
\end{table}

Initially, the existence of an uncrossed edge leaves only one case
for each edge orbit: one for \texttt{Tutte-Coxeter} and six for
\texttt{Byte}. Each resulting case is passed to the SAT solver
with the specified edges required to be crossed or uncrossed.
We split a case further only if the solver reaches the time limit.
One case can thus represent many equivalent crossing patterns.
The cases are independent and can be processed in parallel on
the multi-core server described in \cref{sect:oops}. The graph is
non-1-planar once every resulting case has been proved unsatisfiable.

\begin{theorem}
    \label{thm:cubic30}
    There exist two cubic non-1-planar graphs of order $30$.  In particular,
    the \texttt{Tutte-Coxeter} graph and the \texttt{Byte} graph are non-1-planar.
\end{theorem}

\begin{proof}[Proof (computational)]
    For each graph, we applied the case split described
    above and tested every resulting case with \oops, using the
    separating-cycle propagator from \cref{sect:sepcycle-propagation}. The computations ran up to $48$
    independent solver processes in parallel, one per core, with a time limit
    of $300$ seconds for an individual SAT call; timed-out cases were split
    recursively. The combined running times in \cref{tab:cubic30-runtime}
    include the automorphism computation and all SAT calls.
    Every final case was proved unsatisfiable.

    When crossing constraints are imposed, we disable some rules in the basic
    encoding of \oops, such as fixing a particular relative order between
    twin vertices; these rules need not preserve
    the imposed crossing constraints. We also omit transitivity clauses
    for relative variables, $\sigma$. This only weakens the
    SAT formula, so proving the weakened formula unsatisfiable also
    proves the complete encoding unsatisfiable.
\end{proof}

We also tested separating-cycle propagation (\cref{sect:sepcycle-propagation})
and symmetry-based case splitting (\cref{sect:symmetry-breaking})
on the graphs in \cref{tab:cubic30-runtime} to see their individual impact.
Each optimization reduces running times on
the four additional graphs, but neither resolves \texttt{Tutte-Coxeter}
or \texttt{Byte} even after a week of computation. Combining the two optimizations
resolves both hard instances within hours.

\section{All Subcubic Graphs of Order Less Than 30 Are 1-Planar}
\label{sect:lb}

Our goal in this section is to show that every graph of maximum degree three
and fewer than $30$ vertices is 1-planar. We first observe that the
smallest subcubic non-1-planar graph
is cubic. This is formally proved in \cref{lem:flexibility-core}; the lemma
also proves that this graph is biconnected.
Since cubic graphs have an even order, it suffices to show that all cubic
graphs of order at most $28$ are 1-planar, which is the main result of
the section.

\begin{theorem}
    \label{thm:cubic28}
    Every cubic graph with at most 28 vertices is 1-planar.
\end{theorem}

Proving the claim computationally, that is, exhaustively testing 1-planarity of
all graphs is impractical:
there are $2{,}094{,}480{,}864$ connected cubic graphs with $n=26$ and
$40{,}497{,}138{,}011$ with $n=28$.  Even at ten milliseconds per graph
(which one can achieve with \oops for the \emph{simplest} instances),
the latter family alone would require more than twelve CPU-years.
Harder and larger instances require substantially more time.
\Cref{table:cubic} lists some families of cubic graphs whose 1-planarity was
verified in \cite{oops}, along with the runtimes on a multi-core machine.
For example, testing all cubic graphs with at most $24$ vertices
and all cubic bipartite graphs with at most $28$ vertices took
approximately $1.5$ machine-months in total, averaging less than
a second per instance.
Therefore, directly verifying the claim of \cref{thm:cubic28} is practically unrealistic.

\begin{table}[!t]
    \small
    \centering
    \caption{Computational results for 1-planarity of cubic graphs
        taken from \cite{oops}. The reported runtimes were measured
        in multi-core computations.}
    \label{table:cubic}
    \footnotesize
    \setlength{\tabcolsep}{3pt}
    \begin{tabular*}{\textwidth}{@{\extracolsep{\fill}}rrrrC{60pt}}
        \toprule
        graph class & total count & total runtime & runtime per instance & all $1$-planar? \\
        \midrule
        girth$\geq4$ $n=24$
        & $23,\!780,\!814$ & 10h 20m & $75$ ms & \yes \\
        $n=24$                    & $117,\!940,\!535$ & 6d 10h & $150$ ms & \yes \\
        \cmidrule{1-1}
        bipartite $n=26$             & $245,\!627$ & 25m & $200$ ms & \yes \\
        girth$\ge 5$ $n=26$          & $31,\!478,\!584$ & 3d 6h & $290$ ms & \yes \\
        $n=26$                    & $2,\!094,\!480,\!864$ & & & ? \\
        \cmidrule{1-1}
        bipartite $n=28$          & $2,\!291,\!589$ & 6h 55m & $340$ ms & \yes \\
        girth$\ge 6$ $n=28$          & $4,\!624,\!501$ & 2d 23h & $1700$ ms & \yes \\
        $n=28$                    & $40,\!497,\!138,\!011$ & & & ? \\
        \cmidrule{1-1}
        bipartite $n=30$          & $23,\!466,\!857$ & 3d 9h & $600$ ms & all but one \\
        girth$\ge 7$ $n=30$       & $546$ & 26m & $102$ s & all but two \\
        \bottomrule
    \end{tabular*}
\end{table}

We instead suggest a hybrid approach in which we provide several reduction
rules that reduce the entire search space to a small set of tractable \emph{facts}
about cubic graphs of small orders.  Each such fact can be verified computationally
on a desktop machine within just a few days (\crefrange{clm:flexible22}{clm:girth6-28}, \cref{tab:cubic28-runtime}),
while the reduction rules
(\crefrange{lem:flexibility-core}{lem:flexibility-reduction}) are proved
combinatorially.  This combination reduces the required
computational runtime by a few orders of magnitude.
The central concept of the proof is $k$-flexibility. Recall that, for an integer
$k\geq 0$, a graph $G$ is $k$-flexible if, for every set
$F\subseteq E(G)$ with $|F|\leq k$, it has a 1-planar drawing in which every
edge of $F$ is uncrossed.  In particular, $0$-flexibility is equivalent to
1-planarity.  Before formally describing the reduction rules and establishing
computational facts, let us provide a high-level overview of the proof of
\cref{thm:cubic28}.

For integers $k,n\geq 0$, define the following statement:
\[
\mathcal P_k(n) :=
\qquad
\text{every cubic graph of order at most $n$ is $k$-flexible}.
\]
The claim of \cref{thm:cubic28} is $\mathcal P_0(28)$.
By \cref{lem:flexibility-core} (stated in the next subsection), a smallest counterexample to the claim
would be biconnected, cubic, and triangle-free.
Let $G$ be such a graph of order at most $28$, and consider its girth.
If its girth is five (resp., at least six), then $G$ is 1-planar by
\cref{clm:girth5-28} (resp., \cref{clm:girth6-28}). Hence, we may
assume that $G$ contains a 4-cycle
$C=(v_0,v_1,v_2,v_3)$. Delete $C$ and replace it by adjacent vertices
$a,b$, where $a$ is joined to the external neighbors of $v_0,v_1$ and
$b$ to those of $v_2,v_3$.
The resulting graph $H$ has order at most $26$.  Under $\mathcal P_1(26)$, it has a
1-planar drawing in which $ab$ is uncrossed, which expands to a 1-planar
drawing of $G$.  Thus, proving $\mathcal P_0(28)$ reduces to proving
$\mathcal P_1(26)$.  The same argument, using
\cref{clm:flexible26,clm:flexible24} for cases with girth at least five,
reduces this to $\mathcal P_2(24)$ and $\mathcal P_3(22)$, respectively.
Finally, \cref{clm:flexible22} establishes
$\mathcal P_3(22)$.
Thus, the reductions replace 1-planarity tests on more than 40 billion
graphs with $k$-flexibility tests, for $0\leq k\leq3$, on about
18 million instances of smaller order.

Next in \cref{sect:building} we
provide the utilized reduction rules which serve as the basis of computational
facts discussed in \cref{sect:facts}.

\begin{table}[!t]
    \small
    \centering
    \caption{Computational tests used in the proof of \cref{thm:cubic28}. Running times are rounded
        wall-clock measurements on 48 cores. The total instance count excludes
        classification, while the total running time includes it.}
    \label{tab:cubic28-runtime}
    \setlength{\tabcolsep}{2pt}
    \begin{tabular*}{\textwidth}{@{\extracolsep{\fill}}lllrr@{}}
        \toprule
        result & graph class & test & instances & running time \\
        \midrule
        \cref{clm:flexible22} & girth $\geq 4$, $n\leq 22$ & 3-flexibility
        & $1{,}538{,}493$ & 18 h \\
        \cref{clm:flexible24} & girth $\geq 5$, $n=24$ & 2-flexibility
        & $1{,}620{,}470$ & 32 h \\
        \cref{clm:flexible26} & girth $\geq 5$, $n=26$ & 1-flexibility, cores
        & $4{,}881{,}876$ & 34 h \\
        & & 1-flexibility, fallback
        & $9{,}287$ & $<1$ h \\
        \cref{clm:girth5-28} & girth $=5$, $n=28$ & classification
        &  & 16 h \\
        & & 1-planarity, reduced instances
        & $3{,}788{,}793$ & 18 h \\
        & & 1-planarity, fallback
        & $1{,}522{,}464$ & 35 h \\
        \cref{clm:girth6-28} & girth $\geq 6$, $n=28$ & 1-planarity
        & $4{,}624{,}501$ & 69 h \\
        \midrule
        \multicolumn{3}{l}{Total} & $17{,}985{,}884$ & 9.3 days \\
        \bottomrule
    \end{tabular*}
\end{table}

\subsection{Building Blocks}
\label{sect:building}

\begin{lemma}
    \label{lem:flexibility-core}
    For every integer $k\geq 0$, a smallest subcubic graph that is not
    $k$-flexible is biconnected, cubic, and triangle-free.
\end{lemma}

\begin{proof}
    Let $G$ be a smallest subcubic graph that is not $k$-flexible.  Fix a set
    $F\subseteq E(G)$ with $|F|\leq k$ for which no 1-planar drawing of $G$
    keeps every edge of $F$ uncrossed.

    We first show that $G$ is biconnected.  Suppose otherwise.  If $G$ is
    disconnected, then each of its connected components has smaller order and is
    therefore $k$-flexible.  Draw each component with its edges of $F$
    uncrossed and place the drawings in disjoint regions.  This gives a
    1-planar drawing of $G$ in which every edge of $F$ is uncrossed, a
    contradiction.  Hence $G$ is connected and has a cut vertex.  Every
    block of $G$ has smaller order and is therefore $k$-flexible.  Draw each
    block with its edges of $F$ uncrossed and place the drawings in disjoint
    wedges around each cut vertex. Again, every edge of $F$ is
    uncrossed, a contradiction.  Thus, $G$ is biconnected.

    We next show that $G$ is cubic.  Suppose otherwise.  Since $G$ is
    biconnected, it has a vertex $v$ of degree two.  Let $u$ and $w$ be its
    neighbors, and consider three cases.
    \begin{enumerate}[(a)]
        \item Suppose $uw\notin E(G)$.  Remove $v$ and insert edge $uw$, obtaining
        a smaller subcubic graph $H$.  Let $F_H$ contain every edge of $F$
        outside $\{uv,vw\}$, and include $uw$ if $F$ contains $uv$ or $vw$.
        Then $|F_H|\leq |F|$, so $H$ has a 1-planar drawing in which every
        edge of $F_H$ is uncrossed.  Subdividing $uw$ into the path
        $(u,v,w)$ gives such a drawing of $G$, a contradiction.

        \item Suppose $uw\in E(G)$ and $u,w$ have another common neighbor
        $x$.  If $V(G)=\{u,v,w,x\}$, then $G$ is the planar diamond graph and
        hence is $k$-flexible, a contradiction.  Otherwise $x$ is a cut
        vertex, contradicting biconnectivity.

        \item Suppose $uw\in E(G)$ and $u,w$ have no other common neighbor.
        Contract the triangle $\langle u,v,w\rangle$ to a single vertex, obtaining a
        smaller simple subcubic graph $H$. Let $F_H$ consist of the images
        of edges of $F$ not in the triangle. Since $|F_H|\leq |F|\leq k$,
        minimality gives a 1-planar drawing of $H$ with every edge of
        $F_H$ uncrossed. Expanding the contracted vertex into a triangle
        without adding crossings gives a drawing of $G$ with every edge
        of $F$ uncrossed, a contradiction.
    \end{enumerate}
    Thus, $G$ is cubic.

    Finally, suppose that $G$ contains a triangle $T=\langle u,v,w\rangle$.  Let
    $u',v',w'$ be the respective neighbors of its vertices outside $T$, and
    consider three cases.
    \begin{enumerate}[(a)]
        \item Suppose $u',v',w'$ are pairwise distinct.  Contract $T$ to a
        single vertex, obtaining a smaller simple subcubic graph $H$.
        Let $F_H$ consist of the images of edges of $F$ not in $T$.
        Then $|F_H|\leq |F|$. Expanding the contracted vertex in
        a drawing of $H$ in which every edge of $F_H$ is uncrossed gives a
        drawing of $G$ in which every edge of $F$ is uncrossed, a
        contradiction.

        \item Suppose $u'=v'=w'$.  Then $G=K_4$ is planar and hence
        $k$-flexible, a contradiction.

        \item Otherwise, after relabeling, $u'=v'=x$ and $w'=y\ne x$.  The
        vertices $u,v,w,x$ induce a diamond.  Let $p$ be the third neighbor
        of $x$.  If $p=y$, then $y$ is a cut vertex, contradicting
        biconnectivity.  Hence $p\ne y$.

        Contract the diamond to a single vertex $z$, obtaining a smaller
        simple subcubic graph $H$. Let $F_H$ consist of the images of edges
        of $F$ outside the diamond, including the attachments $xp$ and $wy$
        when they belong to $F$. The resulting
        set $F_H$ has size at most $|F|$.  In a drawing of $H$ in which every
        edge of $F_H$ is uncrossed, expand $z$ inside a small disk into the
        diamond.  This gives a drawing of $G$ in which every edge of $F$ is
        uncrossed, a contradiction.
    \end{enumerate}
    Hence, $G$ is triangle-free.
\end{proof}

By \cref{lem:flexibility-core}, it suffices for the proof of \cref{thm:cubic28} to consider triangle-free cubic
graphs.
The next three lemmas describe reduction rules for instances of girth four
and five:
they replace a 4-cycle or a 5-cycle by a
smaller graph and specify which edges must remain uncrossed so that a drawing
of the reduced graph can be expanded to the original graph.
We define the following \df{path reduction}.
Let \(C=(v_0,\ldots,v_{\ell-1})\) be a chordless
cycle of length \(\ell\in\{4,5\}\) in a cubic graph, let \(x_i\) be the
neighbor of \(v_i\) outside \(C\), and write \(c_i=v_iv_{i+1}\) and
\(s_i=v_ix_i\), with indices taken modulo \(\ell\).  For \(\ell=4\), delete
\(C\), introduce adjacent vertices \(a,b\), and add the edges
\(ax_0,ax_1,bx_2,bx_3\); see \cref{fig:4-cycle-reduction}.
  For \(\ell=5\), delete \(C\), introduce a path
\((a,b,c)\), and add the edges \(ax_0,ax_1,bx_2,cx_3,cx_4\);
see \cref{fig:5-cycle-reduction}.  In either case, the edges corresponding
to $s_0,\ldots,s_{\ell-1}$ are called the \df{attachments} of the reduction.

\begin{lemma}\label{lem:4-cycle-expansion}
    Let \(H\) be a path reduction of a 4-cycle in a cubic graph \(G\) of
    girth four.
    \begin{enumerate}[(i)]
        \item If \(H\) has a 1-planar drawing in which \(ab\) is uncrossed,
        then \(G\) has one in which \(c_0\) and \(c_2\) are uncrossed.

        \item If \(H\) has a 1-planar drawing in which \(ab\) and \(ax_1\) are
        uncrossed, then \(G\) has one in which \(c_0,c_1,c_2\) are uncrossed.
    \end{enumerate}
\end{lemma}

\begin{figure}[!h]
    \centering
    \includegraphics[width=0.9\linewidth]{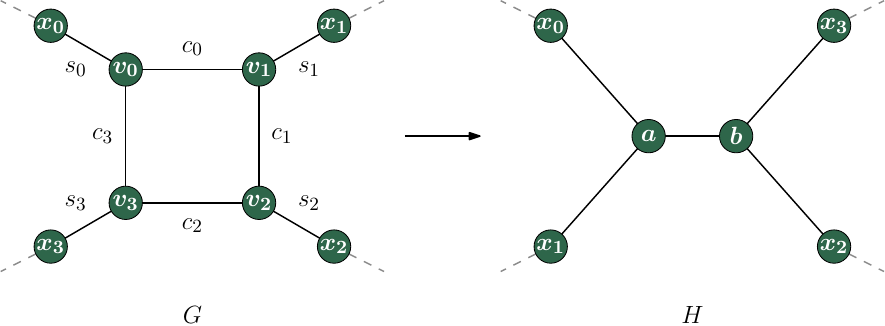}
    \caption{The construction of $H$ from $G$ in
    \cref{lem:4-cycle-expansion}.  The dashed curves indicate the parts of
    the graphs outside the displayed subgraphs.}
    \label{fig:4-cycle-reduction}
\end{figure}

\begin{proof}
    Consider a 1-planar drawing of $H$ satisfying the assumptions of the
    relevant part. We construct a drawing of $G$ by replacing $ab$
    and its endpoints with the 4-cycle, leaving the rest of the
    drawing unchanged.
    Since $G$ is triangle-free, $x_0\ne x_1$ and $x_2\ne x_3$,
    so $H$ is a simple cubic graph. Opposite external neighbors may coincide, but they
    are attached to different vertices of the replacement path.
    Choose a narrow disk around $ab$ and its endpoints
    that contains no crossings and meets the rest of the drawing only
    in the four attachment edges, at distinct boundary points.
    We describe the cyclic order of the attachments by their corresponding
    cycle vertices.
    The two attachments at each endpoint of $ab$ occur consecutively
    around the boundary. Up to cyclic rotation and reflection, the possible
    orders are $(v_0,v_1,v_2,v_3)$ and $(v_0,v_1,v_3,v_2)$.
    Reflection preserves which edges are uncrossed.
    Inside the disk, split $a$ into $v_0,v_1$ and $b$ into $v_2,v_3$,
    keeping the attachment boundary points fixed.

    For part~(i), the left subfigure of \cref{fig:4-cycle-expansion}
    shows how to reconstruct a 1-planar drawing of $G$ from the
    drawing of $H$ for each cyclic order. The first drawing has no
    crossing, while the second has exactly one crossing, between $c_1$
    and $c_3$. No attachment gains a crossing, so the
    resulting drawing is 1-planar and $c_0,c_2$ remain uncrossed.

    For part~(ii), the first drawing again applies without a
    crossing. For the second cyclic order, the right subfigure shows how
    to reconstruct a 1-planar drawing of $G$ from the drawing of $H$,
    using only a crossing between $c_3$ and $s_1$. The edge $ax_1$, which
    becomes $s_1$, is uncrossed outside the disk by assumption.  Thus the
    resulting drawing is 1-planar and $c_0,c_1,c_2$ remain uncrossed.
\end{proof}

\begin{figure}[!t]
    \centering
    \begin{subfigure}[t]{0.573\linewidth}
        \centering
        \includegraphics[width=\linewidth]{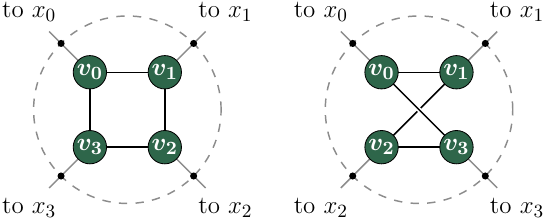}
        \caption*{part~(i)}
    \end{subfigure}
    \hfill
    \begin{subfigure}[t]{0.267\linewidth}
        \centering
        \includegraphics[width=\linewidth]{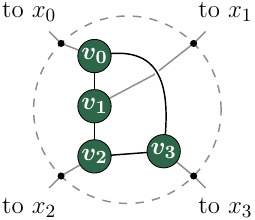}
        \caption*{part~(ii)}
    \end{subfigure}
    \caption{Reconstructing the 4-cycle for \cref{lem:4-cycle-expansion}:
    part~(i) on the left and part~(ii) on the right. Black dots mark
    where attachments meet the dashed disk boundaries.}
    \label{fig:4-cycle-expansion}
\end{figure}

\begin{lemma}\label{lem:5-cycle-path-expansion}
    Let $H$ be a path reduction of a 5-cycle in a cubic graph $G$ of
    girth five.  If $H$ has a 1-planar drawing in which $ab$, $bc$, and
    $bx_2$ are uncrossed, then $G$ is 1-planar.
\end{lemma}

\begin{figure}[h]
    \centering
    \includegraphics[width=0.9\linewidth]{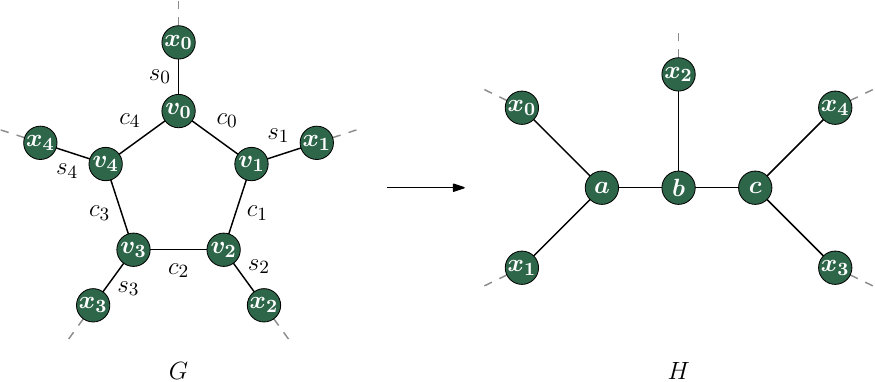}
    \caption{The 5-cycle-to-path reduction from $G$ to $H$.  The dashed
    curves indicate the parts of the graphs outside the displayed
    subgraphs.}
    \label{fig:5-cycle-reduction}
\end{figure}

\begin{proof}
    Consider a 1-planar drawing of $H$ in which $ab$, $bc$, and $bx_2$
    are uncrossed. We reconstruct a drawing of $G$ inside a small disk
    around the path $(a,b,c)$, leaving the drawing outside the disk
    unchanged. Choose the disk to contain no crossings and to meet the
    five attachments at distinct boundary points.
    The attachments corresponding to $v_0,v_1$ occur consecutively around
    the boundary, as do those corresponding to $v_3,v_4$.
    Up to cyclic rotation and reflection, there are four distinct boundary orders.
    Reversing the path, with the relabeling $a\leftrightarrow c$,
    $v_0\leftrightarrow v_4$, and $v_1\leftrightarrow v_3$, identifies
    two of these orders and preserves the assumption that $bx_2$ is uncrossed.
    \Cref{fig:5-cycle-path-expansion} shows how to reconstruct a 1-planar
    drawing of $G$ from the drawing of $H$ for the three remaining orders.
    The only attachment that can gain a crossing is $s_2=v_2x_2$.
    It corresponds to $bx_2$, which is uncrossed in the drawing of $H$.
    Thus every edge is crossed at most once.
\end{proof}

\begin{figure}[!t]
    \centering
    \includegraphics[width=\linewidth]{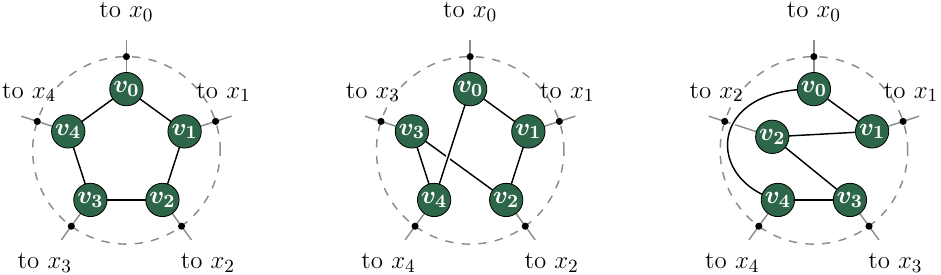}
    \caption{Reconstructing the 5-cycle $(v_0,\ldots,v_4)$ inside a disk
    for the three attachment orders in \cref{lem:5-cycle-path-expansion}.
    From left to right: the drawings have no crossing, a crossing between
    $v_2v_3$ and $v_4v_0$, and a crossing between $v_4v_0$ and the attachment
    toward $x_2$. Black dots mark where attachments meet the dashed disk boundaries.}
    \label{fig:5-cycle-path-expansion}
\end{figure}

We also define a \df{vertex reduction} as follows.
Let $C=(v_0,\ldots,v_4)$ be a 5-cycle in a cubic graph of girth five,
let $x_i$ be the neighbor of $v_i$
outside $C$, and write $c_i=v_iv_{i+1}$ and $s_i=v_ix_i$, with indices
modulo five.  The vertex reduction of $C$ is the graph obtained by contracting
$C$ to one vertex; its attachments are the five edges corresponding to
$s_0,\ldots,s_4$.

\begin{lemma}\label{lem:5-cycle-star-expansion}
    Let $H$ be the vertex reduction of a 5-cycle $C$ in a cubic graph $G$
    of girth five.
    \begin{enumerate}[(i)]
        \item If $H$ has a 1-planar drawing in which all five attachments
        are uncrossed, then, for every $i$, $G$ has a 1-planar drawing in
        which the three edges incident with $v_i$ are uncrossed.

        \item Let $e\in E(H)$.  If $H$ has a 1-planar drawing in which $e$
        and at least three attachments different from $e$ are uncrossed,
        then $G$ has a 1-planar drawing in which the edge corresponding
        to $e$ is uncrossed.
    \end{enumerate}
\end{lemma}

\begin{proof}
    Consider a 1-planar drawing of $H$ satisfying the assumptions of the
    relevant part. We reconstruct a drawing of $G$ by replacing the
    contracted vertex with the 5-cycle inside a small disk, leaving the
    drawing outside the disk unchanged. Choose the disk to contain no
    crossings and to meet the five attachments at distinct boundary points.
    Up to cyclic rotation and reflection, there are $(5-1)!/2=12$
    possible boundary orders, listed in \cref{tab:5-cycle-star-expansion}
    using the corresponding cycle vertices.
    \Cref{fig:5-cycle-vertex-expansion} shows drawings realizing the listed
    crossings for two boundary orders in both parts; the remaining cases
    are analogous.

    For part~(i), relabel the cycle so that the chosen vertex is $v_0$.
    It suffices to keep the three edges $c_4,c_0,s_0$ incident with $v_0$
    uncrossed. For every boundary order, there is a drawing inside the disk
    with that order and the crossings listed in the second column of
    \cref{tab:5-cycle-star-expansion}. These crossings leave $c_4,c_0,s_0$
    uncrossed, and no edge occurs in more than one listed crossing.
    Every attachment is uncrossed outside the disk by assumption.
    Thus the resulting drawing is 1-planar and proves part~(i).

    For part~(ii), assume that $H$ has a 1-planar drawing in which $e$ and
    three attachments different from $e$ are uncrossed. Among the three
    corresponding vertices of $C$, two must be consecutive, since a 5-cycle
    has no independent set of size three. Relabel the cycle so that these
    two attachments correspond to $s_0,s_1$. For every boundary order,
    there is a drawing inside the disk with that order and the crossings
    listed in the third column of \cref{tab:5-cycle-star-expansion}.
    No edge occurs in more than one listed crossing, and only
    $s_0,s_1$ may gain crossings among the attachments. Since these are
    uncrossed outside the disk, the resulting drawing is 1-planar.
    Moreover, the local drawing does not cross $e$: if $e$ is not an
    attachment, it lies outside the disk, and if it is an attachment, it is
    different from the attachments corresponding to $s_0,s_1$.
    Thus the edge corresponding to $e$ remains uncrossed.
\end{proof}

\begin{table}[!t]
    \centering
    \caption{Crossing patterns for the local expansions in
    \cref{lem:5-cycle-star-expansion}. The second column keeps
    $c_4,c_0,s_0$ uncrossed as required in part~(i); in the third column,
    only the attachments $s_0,s_1$ participate in crossings, as required
    in part~(ii). Here $\varnothing$ denotes no crossings.}
    \label{tab:5-cycle-star-expansion}
    \begin{tabular*}{\textwidth}{@{\extracolsep{\fill}}>{$}c<{$}>{$}l<{$}>{$}l<{$}@{}}
        \toprule
        \multicolumn{1}{c}{cyclic order} &
        \multicolumn{1}{c}{crossing pattern, part~(i)} &
        \multicolumn{1}{c}{crossing pattern, part~(ii)} \\
        \midrule
        (v_0,v_1,v_2,v_3,v_4)&\varnothing&\varnothing\\
        (v_0,v_1,v_2,v_4,v_3)&c_2\mathbin\times s_4&c_2\mathbin\times c_4\\
        (v_0,v_1,v_3,v_2,v_4)&c_1\mathbin\times c_3&c_1\mathbin\times c_3\\
        (v_0,v_1,v_3,v_4,v_2)&c_1\mathbin\times c_3,s_2\mathbin\times s_4&
               c_1\mathbin\times s_0,c_2\mathbin\times c_4\\
        (v_0,v_1,v_4,v_2,v_3)&c_1\mathbin\times s_4&
               c_1\mathbin\times c_3,c_4\mathbin\times s_1\\
        (v_0,v_1,v_4,v_3,v_2)&c_1\mathbin\times s_4,c_3\mathbin\times s_2&
               c_1\mathbin\times c_4\\
        (v_0,v_2,v_1,v_3,v_4)&c_2\mathbin\times s_1&c_0\mathbin\times c_2\\
        (v_0,v_2,v_1,v_4,v_3)&c_2\mathbin\times s_1,s_3\mathbin\times s_4&
               c_2\mathbin\times s_0\\
        (v_0,v_2,v_3,v_1,v_4)&c_3\mathbin\times s_1&c_3\mathbin\times s_1\\
        (v_0,v_2,v_4,v_1,v_3)&
               \begin{array}[t]{@{}l@{}}
                   c_1\mathbin\times c_3,s_1\mathbin\times s_3,s_2\mathbin\times s_4\\[-2pt]
                   \text{\footnotesize Figure~\ref{fig:5-cycle-vertex-expansion}(a)}
               \end{array}&
               \begin{array}[t]{@{}l@{}}
                   c_1\mathbin\times c_4,c_2\mathbin\times s_0,c_3\mathbin\times s_1\\[-2pt]
                   \text{\footnotesize Figure~\ref{fig:5-cycle-vertex-expansion}(c)}
               \end{array}\\
        (v_0,v_3,v_1,v_2,v_4)&
               \begin{array}[t]{@{}l@{}}
                   c_1\mathbin\times c_3,s_1\mathbin\times s_3\\[-2pt]
                   \text{\footnotesize Figure~\ref{fig:5-cycle-vertex-expansion}(b)}
               \end{array}&
               \begin{array}[t]{@{}l@{}}
                   c_0\mathbin\times c_2,c_3\mathbin\times s_0\\[-2pt]
                   \text{\footnotesize Figure~\ref{fig:5-cycle-vertex-expansion}(d)}
               \end{array}\\
        (v_0,v_3,v_2,v_1,v_4)&c_1\mathbin\times s_3,c_3\mathbin\times s_1&
               c_0\mathbin\times c_3\\
        \bottomrule
    \end{tabular*}
\end{table}

\begin{figure}[!t]
    \centering
    \begin{subfigure}[t]{0.47\linewidth}
        \centering
        \includegraphics[width=\linewidth]{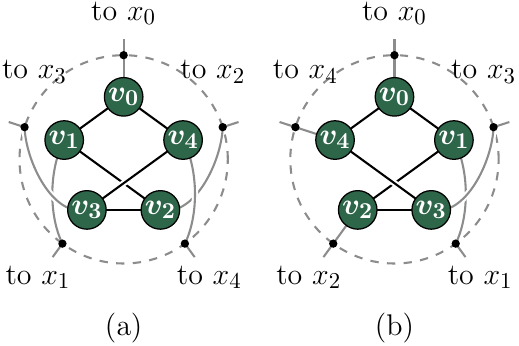}
        \caption*{part~(i)}
    \end{subfigure}\hfill%
    \begin{subfigure}[t]{0.47\linewidth}
        \centering
        \includegraphics[width=\linewidth]{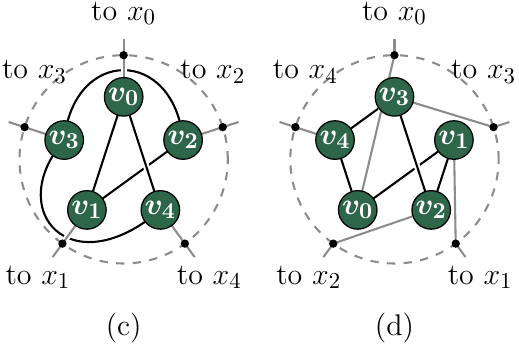}
        \caption*{part~(ii)}
    \end{subfigure}
    \caption{Reconstructing the 5-cycle $(v_0,\ldots,v_4)$ inside a disk
    in \cref{lem:5-cycle-star-expansion}. Drawings (a,b) keep all three
    edges incident with $v_0$ uncrossed. In drawings (c,d), only the
    attachments toward $x_0$ and $x_1$ may gain crossings.
    Black dots mark where attachments meet the dashed disk boundaries;
    each edge is crossed at most once.}
    \label{fig:5-cycle-vertex-expansion}
\end{figure}

As the final result of this section, we show how to reduce $\mathcal P_k(n)$
to $\mathcal P_{k+1}(n-2)$.

\begin{lemma}\label{lem:flexibility-reduction}
    Let $n\geq4$ be even and let $k\in\{0,1,2\}$. Assume that
    \begin{enumerate}[(a)]
        \item every cubic graph of order at most $n-2$ is
        $(k+1)$-flexible, that is, $\mathcal P_{k+1}(n-2)$ holds, and
        \item every biconnected, nonplanar cubic graph of order $n$ and
        girth at least five is $k$-flexible.
    \end{enumerate}
    Then every cubic graph of order at most $n$ is $k$-flexible, that is,
    $\mathcal P_k(n)$ holds.
\end{lemma}

\begin{proof}
    For the sake of contradiction, suppose that $\mathcal P_k(n)$ is false.
    Among all subcubic graphs that
    are not $k$-flexible, choose one, say $G$, of minimum order.
    Then $G$ is nonplanar and $|V(G)|\leq n$. By \cref{lem:flexibility-core},
    $G$ is biconnected, cubic, and
    triangle-free.  Moreover, $\mathcal P_{k+1}(n-2)$ implies
    $\mathcal P_k(n-2)$.  Since cubic graphs have even order, $G$ therefore
    has order exactly $n$. Furthermore, $G$ contains a 4-cycle: otherwise,
    since $G$ is triangle-free, assumption~(b) would make it $k$-flexible.

    Let $C$ be a 4-cycle in $G$, and choose a set $F\subseteq E(G)$ of
    at most $k$ edges that cannot all be kept uncrossed in a 1-planar
    drawing of $G$. Suppose first that $F$ does not consist of two
    adjacent edges of $C$. Since $|F|\leq2$, label the cycle so that
    $F\cap E(C)\subseteq\{c_0,c_2\}$, and take the corresponding path
    reduction $H$. Since $G$ is triangle-free, $H$ is a simple cubic
    graph of order $n-2$.
    Since $H$ is $(k+1)$-flexible, it has a 1-planar drawing in which
    $ab$ and the edges corresponding to $F\setminus E(C)$ are uncrossed:
    there are at most $k+1$ such edges. Part~(i) of
    \cref{lem:4-cycle-expansion} reconstructs a drawing of $G$ in which
    every edge of $F$ is uncrossed, a contradiction.

    Finally, suppose that $F$ consists of two adjacent edges of $C$, so $k=2$.
    Label the vertices of $C$ so that these edges are $c_0=v_0v_1$
    and $c_1=v_1v_2$, and consider the corresponding path reduction $H$.
    Since $H$ is $(k+1)$-flexible, it has a 1-planar drawing in which
    $ab$ and $ax_1$ are uncrossed. Part~(ii) of \cref{lem:4-cycle-expansion} again
    produces a drawing of $G$ in which every edge of $F$ is uncrossed, a
    contradiction.
\end{proof}

\subsection{Verification}
\label{sect:facts}

To prove \cref{thm:cubic28}, we establish five computational facts about
biconnected, nonplanar cubic graphs.  The computations for
\cref{clm:flexible22,clm:flexible24,clm:girth6-28} tested the relevant graphs
directly. For each instance, we found a crossing pattern that avoids any
edges required to remain uncrossed and checked the planarity of the
resulting graph. The computations for \cref{clm:flexible26,clm:girth5-28} first
applied the 5-cycle reductions and directly tested the instances not resolved by the
reduced graphs. The numbers of tested instances and the corresponding
running times are summarized in \cref{tab:cubic28-runtime}.
We generated the graph collections
using \texttt{Minibaum}~\cite{Brinkmann96} and \texttt{nauty}~\cite{nauty}.

\begin{fact}\label{clm:flexible22}
    Every biconnected, nonplanar cubic graph with $n\leq 22$ and girth $\geq 4$
    is $3$-flexible.
\end{fact}

\begin{proof}[Proof (computational)]
    We tested all $1{,}538{,}493$ nonisomorphic biconnected, nonplanar
    cubic graphs of order at most $22$ and girth at least four.
    For every such graph $G$ and every set $F\subseteq E(G)$ with
    $|F|\leq3$, the computation found a 1-planar drawing of $G$ in which
    every edge of $F$ is uncrossed.
\end{proof}

\begin{fact}\label{clm:flexible24}
    Every biconnected, nonplanar cubic graph with $n=24$ and girth $\geq 5$
    is $2$-flexible.
\end{fact}

\begin{proof}[Proof (computational)]
    We tested all $1{,}620{,}470$ nonisomorphic biconnected, nonplanar
    cubic graphs of order $24$ and girth at least five.
    For every such graph $G$ and every set $F\subseteq E(G)$ with
    $|F|\leq2$, the computation found a 1-planar drawing of $G$ in which
    every edge of $F$ is uncrossed.
\end{proof}

\begin{fact}\label{clm:flexible26}
    Every biconnected, nonplanar cubic graph with $n=26$ and girth $\geq 5$
    is $1$-flexible.
\end{fact}

\begin{proof}[Proof (computational)]
    There are $31{,}478{,}465$ nonisomorphic biconnected, nonplanar cubic graphs with
    $n=26$ and girth at least five.
    We first consider the $181{,}227$ graphs of girth at least six,
    which we tested directly.
    For every such graph $G$ and every edge $e\in E(G)$, the computation
    found a 1-planar drawing of $G$ in which $e$ is uncrossed.

    It remains to consider graphs of girth five. Since there are
    $31{,}297{,}238$ such graphs, testing their $1$-flexibility directly
    would be too expensive. Instead, we applied vertex reductions to every
    graph and grouped instances that produced isomorphic reduced graphs.
    By \cref{lem:5-cycle-star-expansion}, suitable drawings
    of one reduced graph prove the $1$-flexibility of every graph obtained by
    expanding its degree-five vertex into a 5-cycle. We tested the reduced
    graph shared by each group; if the required drawings were not found,
    we tested all possible expansions directly. Next we discuss the process
    in detail.

    For each graph $G$ of girth five, we chose a 5-cycle $C$, applied the vertex reduction
    to $C$, and denoted the resulting order-22 graph by $H$. Since $G$ has
    girth five, the five neighbors of the contracted vertex are distinct.
    After removing duplicate reduced graphs, we obtained $4{,}700{,}649$
    distinct graphs, which we call \df{cores}.
    Fix a core $H$, and let $S\subseteq E(H)$ be the set of its five edge
    attachments.  A \df{parent} of $H$ is a cubic graph obtained by replacing the
    degree-five vertex of $H$ with a 5-cycle and attaching the edges in $S$
    bijectively to its vertices.
    By construction, every original graph $G$ is a parent of one of the
    cores.

    We now show that every parent of a core $H$ is $1$-flexible:
    for each edge, we need a 1-planar drawing in which it is uncrossed.
    Each edge either belongs to the inserted 5-cycle or is inherited from $H$.

    \begin{enumerate}[(a)]
        \item Consider an edge of the inserted 5-cycle.  By part~(i) of
        \cref{lem:5-cycle-star-expansion}, it suffices to find a 1-planar
        drawing of $H$ in which all five attachments are uncrossed.  The
        computation searched for a crossing pattern $M$ such that $H_M$ is
        planar and $M$ avoids $S$.

        \item Consider the edge corresponding to $e\in E(H)$. By part~(ii)
        of \cref{lem:5-cycle-star-expansion}, it suffices to find a 1-planar
        drawing of $H$ in which $e$ and at least three attachments other than
        $e$ are uncrossed. For every $e\in E(H)$, the computation searched
        for a crossing pattern $M$ such that $H_M$ is planar, $M$ avoids $e$,
        and $M$ avoids at least three attachments different from $e$.
    \end{enumerate}

    We allowed $120$ seconds to process each core. For most cores, the solver
    found all the required patterns within this limit; see
    \cref{tab:cubic28-runtime}. For a few cores, however, the computation
    timed out before finding all required patterns. Such timeouts may occur when a required
    pattern does not exist. Note, however, that the conditions in
    \cref{lem:5-cycle-star-expansion} are sufficient but not necessary for
    the parents to be $1$-flexible. We therefore tested the parents of these
    cores directly; these tests are listed as \emph{fallbacks} in the table. Up to
    rotation and reflection, there are $(5-1)!/2=12$ cyclic orders of the
    five attachments around the inserted 5-cycle.  The computation
    constructed a parent for each of these 12 orders and verified the $1$-flexibility of every
    resulting parent that is biconnected, nonplanar, and has girth at least
    five.

    Hence every graph in the input family is $1$-flexible.
\end{proof}

\begin{fact}\label{clm:girth5-28}
    Every biconnected, nonplanar cubic graph with $n=28$ and girth $5$ is
    1-planar.
\end{fact}

\begin{proof}[Proof (computational)]
    There are $652{,}157{,}758$ nonisomorphic biconnected, nonplanar cubic
    graphs of order $28$ and girth five. Since this family is too large
    to test directly, we applied the reductions described below.
    Each reduction produces a graph with fewer vertices, and we explain
    how suitable drawings of that graph yield a 1-planar drawing of the
    original graph. Only a small fraction of the instances remained
    unresolved by these reductions; we tested them directly.

    Fix a graph $G$ in the family. For each 5-cycle $C=(v_0,\ldots,v_4)$
    of $G$, consider its path reductions; see \cref{fig:5-cycle-reduction}.
    Let $H$ denote a resulting graph and $(a,b,c)$ its replacement path.
    Since $G$ has girth five, $H$ is a simple cubic graph of order $26$.
    We consider the following reductions in order.
    \begin{enumerate}[(a)]
        \item If, for some $H$, a second 5-cycle $D$ containing the entire
        path $(a,b,c)$ has a path reduction $H'$ that is biconnected and has girth
        at least five, write $(p,q,r)$ for its replacement path and require
        that the edge incident with $b$ outside $D$ becomes incident with $q$.
        Assign $G$ to one such \emph{double 5-cycle reduction}, if one exists.
        The graph $H'$ is cubic and has order $24$. If it is planar, it has
        a drawing in which $pq$ and $qr$ are uncrossed; otherwise,
        \cref{clm:flexible24} gives a 1-planar drawing with both edges
        uncrossed.

        We reconstruct a drawing of $G$ from this drawing of $H'$ as follows.
        The two vertices of $D$ other than $a,b,c$ correspond to vertices
        of $G$ outside $C$. Since $G$ has
        girth five, they cannot attach to consecutive vertices of $C$.
        By rotating or reversing the labels of $C$, we may therefore assume
        that the subgraph to be restored consists of the cycle $C$,
        vertices $x_0,x_3$, and edges $v_0x_0,v_3x_3,x_0x_3$.
        Choose a small disk around the uncrossed path $(p,q,r)$ that
        contains no crossings and meets its five attachments at distinct
        boundary points. Up to cyclic rotation and reflection, these
        attachments have the four boundary orders shown in
        \cref{fig:overlapping-5-cycle-expansion}. The figure shows how to
        reconstruct the seven-vertex subgraph inside the disk for each
        order, leaving the drawing outside unchanged. Every edge of the
        restored subgraph is crossed at most once, and no attachment gains
        a crossing. Thus the resulting drawing of $G$ is 1-planar.

        \begin{figure}[!t]
            \centering
            \includegraphics[width=\linewidth]{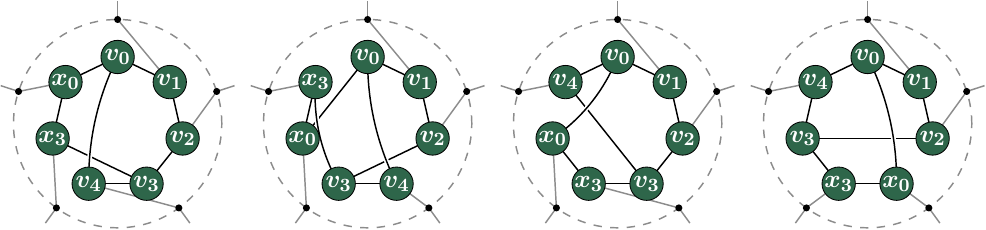}
            \caption{Reconstructing the cycle $(v_0,\ldots,v_4)$ and the
            path $(v_0,x_0,x_3,v_3)$ after the double 5-cycle reduction,
            for the four attachment orders. Each edge is crossed at most
            once inside the disk, and no attachment gains a crossing. Black dots
            mark where attachments meet the dashed disk boundaries.}
            \label{fig:overlapping-5-cycle-expansion}
        \end{figure}

        \item If no double 5-cycle reduction applies, search all the graphs
        $H$ for a 5-cycle $D$ of the form
        \(
            D=(b,a,u_0,u_1,m),
        \)
        where $au_0$ and $bm$ are the edges of $D$ incident with $a$ and
        $b$, respectively, that do not belong to $(a,b,c)$.  Contract
        $V(D)\cup\{c\}$ to a single vertex $z$.  If this produces a simple
        connected graph, we call it a quotient. The
        quotient has order 21, the vertex $z$ has degree six, and every
        other vertex has degree three. If such a degree-six quotient exists,
        assign $G$ to one; see \cref{fig:fact4-overlap-quotients}.

        To establish 1-planarity of $G$, it suffices to find a 1-planar
        drawing of the quotient in which all six edges incident with $z$
        are uncrossed. Choose a small disk around $z$ that contains no
        crossings and meets these edges at distinct boundary points.
        Consider their cyclic order around the disk boundary.
        Up to cyclic rotation and reflection, there are $(6-1)!/2=60$
        boundary orders. For each order, we found a 1-planar drawing
        inside the disk of the subgraph of $G$ removed by the two reductions,
        with its six attachments in that order and uncrossed.
        Inserting the corresponding drawing leaves the drawing outside
        the disk unchanged and yields a 1-planar drawing of $G$.

        \item If neither preceding construction applies, repeat the quotient
        construction using a 6-cycle $D=(b,a,u_0,u_1,u_2,m)$.
        The quotient must again be simple and connected; it has order 20,
        the vertex $z$ has degree seven, and every other vertex has degree
        three. If such a degree-seven quotient exists, assign $G$ to one;
        see \cref{fig:fact4-overlap-quotients}.

        As in the degree-six case, we sought a 1-planar drawing of the
        quotient in which all edges at $z$ are uncrossed.
        Here, however, we found reconstructions for only some boundary orders. Up to cyclic
        rotation and reflection, there are $(7-1)!/2=360$ boundary orders.
        For 355 orders, we found a 1-planar
        drawing of the restored subgraph inside the disk, with its seven
        attachments uncrossed and in the specified order.

        The correspondence between the seven edges at $z$ and the seven
        attachments of the restored subgraph may differ among graphs
        assigned to the same quotient. For every such correspondence,
        we searched for a 1-planar drawing of the quotient in which all
        seven edges at $z$ are uncrossed and their boundary order is one
        of these 355 orders. Whenever such a drawing was found, inserting
        the corresponding local drawing left the drawing outside the disk
        unchanged and yielded a 1-planar drawing of the assigned graph $G$.

        \item If none of the preceding constructions applies, choose one
        path reduction $H$ of a 5-cycle of $G$, with replacement path
        $(a,b,c)$. Group together graphs whose reduced pairs $(H,b)$ are
        isomorphic, with the isomorphism preserving the distinguished vertex.
        The pair $(H,b)$ identifies this
        group and specifies the test shared by its graphs.

        For each group, we searched for a 1-planar drawing of $H$ in which
        the three edges incident with $b$ are uncrossed. For every graph
        $G$ in the group, these edges correspond to $ab$, $bc$, and $bx_2$.
        Whenever such a drawing was found, \cref{lem:5-cycle-path-expansion}
        yielded a 1-planar drawing of every graph in the group by restoring
        its original 5-cycle.
    \end{enumerate}

\begin{figure}[!t]
    \centering
    \includegraphics[width=0.95\linewidth]{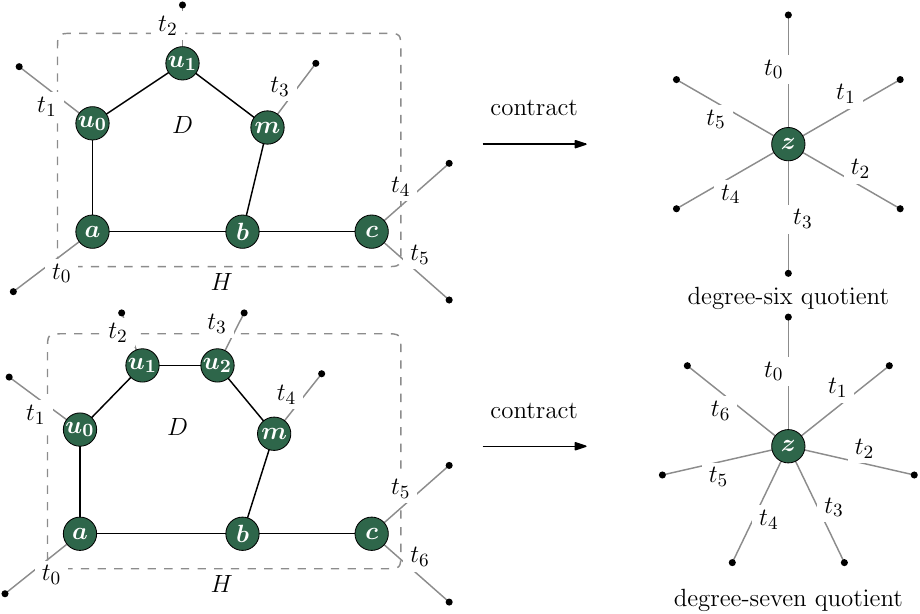}
    \caption{The overlap quotients used in the proof of
        \cref{clm:girth5-28}.  Each dashed region is contracted to $z$; the
        edges $t_i$ join the region to the rest of the graph.  In a drawing of
        the quotient, the cyclic order of these edges around $z$ is its boundary
        order.}
    \label{fig:fact4-overlap-quotients}
\end{figure}

    The double reduction (a) handled $254{,}354{,}420$ graphs. We assigned
    $350{,}383{,}009$ graphs to degree-six quotients (b), $45{,}970{,}275$
    to degree-seven quotients (c), and the remaining $1{,}450{,}054$ to
    pairs $(H,b)$ in case (d). After removing isomorphic duplicates from
    the latter three groups, preserving the distinguished vertex in
    pairs $(H,b)$, we obtained $3{,}788{,}793$
    \emph{reduced instances}. For each instance, we searched for the
    drawing or drawings required by its case above.
    We allowed $60$ seconds per reduced instance. If the computation
    timed out before finding all required drawings, we tested every
    original graph assigned to that instance directly. These
    $1{,}522{,}464$ graphs are listed as \emph{fallbacks} in
    \cref{tab:cubic28-runtime}; the computation found a 1-planar drawing
    of each. Thus every graph in the input family is 1-planar.
\end{proof}

\begin{fact}\label{clm:girth6-28}
    Every biconnected, nonplanar cubic graph with $n=28$ and girth $\geq 6$ is
    1-planar.
\end{fact}

\begin{proof}[Proof (computational)]
    We tested all $4{,}624{,}501$ nonisomorphic biconnected, nonplanar
    cubic graphs of order $28$ and girth at least six.
    The computation found a 1-planar drawing of every such graph.
\end{proof}

Now we are ready to complete the proof of the main result in the section.

\begin{proof}[Proof of~\cref{thm:cubic28}]
    Recall that $\mathcal P_k(n)$ denotes the statement that every
    cubic graph of order at most $n$ is $k$-flexible.
    Next we establish $\mathcal P_3(22)$, $\mathcal P_2(24)$,
    $\mathcal P_1(26)$, and $\mathcal P_0(28)$ in this order.

    \begin{itemize}
        \item \emph{$\mathcal P_3(22)$.}
        Since every planar graph is $3$-flexible,
        \cref{lem:flexibility-core,clm:flexible22} imply
        $\mathcal P_3(22)$.

        \item \emph{$\mathcal P_2(24)$.}
        Applying \cref{lem:flexibility-reduction} with
        $\mathcal P_3(22)$ and \cref{clm:flexible24} gives
        $\mathcal P_2(24)$.

        \item \emph{$\mathcal P_1(26)$.}
        Applying \cref{lem:flexibility-reduction} with
        $\mathcal P_2(24)$ and \cref{clm:flexible26} gives
        $\mathcal P_1(26)$.

        \item \emph{$\mathcal P_0(28)$.}
        Applying \cref{lem:flexibility-reduction} with
        $\mathcal P_1(26)$ and \cref{clm:girth5-28,clm:girth6-28} gives
        $\mathcal P_0(28)$.
    \end{itemize}

    The last statement is exactly \cref{thm:cubic28}, since
    $0$-flexibility is equivalent to 1-planarity.
\end{proof}

\section{Discussion}

Are the Tutte-Coxeter and Byte graphs \emph{the only} cubic non-1-planar
graphs of order 30? We conjecture that they are. Our experiments suggest
that non-1-planarity is more common among cubic graphs of higher girth.
All 546 connected cubic graphs of this order and girth at least seven
have been verified: these two graphs are the only non-1-planar
instances among them. However, extending this classification to all
cubic graphs of order 30 remains challenging. A direct exhaustive test
would involve approximately 845 billion connected graphs, and the tools
developed in \cref{sect:lb} are not sufficient to make such a test practical.

Finally, we emphasize that
the computational complexity of recognizing \emph{cubic} 1-planar graphs remains open.
For comparison, recognizing 1-planar graphs is \NP-complete for graphs
of maximum degree $20$~\cite{CM13}.
Our results might provide useful building blocks for a hardness proof
for recognizing cubic 1-planar graphs.

\paragraph{Acknowledgements}
The author thanks Michael Kaufmann for pointing out in a private communication
that cube-connected cycles provide examples of cubic non-1-planar graphs.

\bibliographystyle{abbrv}
\bibliography{cubic}

\end{document}